\documentclass{article}

\usepackage{amsmath, amsthm, amsfonts, amssymb, enumerate}

\usepackage{xcolor}

\newtheorem{thm}{Theorem}[section]

\theoremstyle{definition}

\theoremstyle{remark}
\newtheorem{remark}[thm]{Remark}

\newcommand{\diff}{\mathrm{d}}
\newcommand{\sign}{\operatorname{sign}}
\newcommand{\abs}[1]{\left|#1\right|}\title{Skeleton for the one-dimensional aggregation equation}
\author{Juan Carlos Cantero and Joan Orobitg\\
  \small Departament de Matem\`atiques, Universitat Aut\`onoma de Barcelona
}

\begin{document}
\maketitle

\begin{abstract}
For the aggregation equation in $\mathbb{R}$, we consider the evolution of an initial density corresponding to the characteristic function of some set $\Omega_0$. We study the limit measure at the blow up time 1 for $\Omega_0$ open or compact and we inspect the limit set (skeleton) where this measure is supported.\\
\textbf{MSC classes:} 35Q35 (Primary) 35F20, 35B30 (Secondary)

\end{abstract}

\maketitle

\section{Introduction}
Consider the aggregation equation
\begin{equation}
\begin{cases}
\rho_t + \mathrm{div}(\rho v) =0,\\
v(\cdot,t) = -\nabla N \ast \rho(\cdot,t),\\
\rho(\cdot,0)=\rho_0,
\end{cases}
\label{jc100502}
\end{equation}
where $N$ is the fundamental solution of the laplacian and $\rho_0$ is a bounded compactly supported function. This problem has been a very active area of research in the literature (see \cite[Section 1]{bll}).

As in \cite{bll}, we can compute an explicit solution to \eqref{jc100502} along the flow map $X(\cdot,t)$ as
\begin{equation}
\rho(X(\alpha,t),t) = \left(\frac{1}{\rho_0(\alpha)}- t \right)^{-1} =\frac{\rho_0(\alpha)}{1-t\rho_0(\alpha)}, 
\label{solnflow}
\end{equation}
where $X(\cdot,t)$ is defined, as usual, by the ODE 
\begin{equation}\label{trajsoln}
\begin{cases} \frac{\diff}{\diff t}X(\alpha,t) = v(X(\alpha,t),t),\\ X(\alpha,0)=\alpha. 
\end{cases} 
\end{equation}

We are interested in the case of the evolution of an aggregation patch, that is, when $\rho_0=\chi_{\Omega_0}$ for some bounded domain $\Omega_0$. In this case the explicit solution can be simply written, according to \eqref{solnflow}, as
\begin{equation}
\rho(\cdot,t)= \frac{1}{1-t}\chi_{\Omega_t},
\label{patchsolution}
\end{equation}
where $\Omega_t = X(\Omega_0,t)$, $0\le t < 1$. That is, the set $\Omega_t$ is the evolution of $\Omega_0$ along $X(\cdot,t)$. With this expression, it is clear that the blow-up occurs at time $t=1$. It is well known that equation \eqref{jc100502} preserves the $L^1$ norm of the scalar $\rho$, and hence
$$\|\rho(\cdot,t)\|_{L^1} = \frac{1}{1-t} \left|\Omega_t\right|  = \|\rho_0\|_{L^1} = \left|\Omega_0\right|.$$
Therefore
$$\left|\Omega_t\right| = (1-t)\left|\Omega_0\right| \to 0 \,\,\,\,\,\,\text{as }\,\,\,t \to 1. $$

In \cite{bglv}, they show that for an initial data of the form $\rho_0=\chi_{\Omega_0}$ with $\Omega_0$ a domain of class $C^{1,\gamma}$ then the domain $\Omega_t$ in the solution \eqref{patchsolution} is also of class $C^{1,\gamma}$ when $t<1$.

A challenging question is to study in detail the structure of the \textit{skeleton}, that is, the blow-up set 
$$\Omega_1 = \lim_{t\to 1^-} X(\Omega_0,t).$$

Explicit computations  in \cite[Section 4.1]{bll} shows that an elliptical patch collapses at time 1 to a measure supported in an interval contained in the straight line defined by the major axis. On the other hand, numerical simulations in \cite[Sections 4.2-4.4]{bll} suggest a much more cumbersome behaviour for other patches not that regular, even though they collapse into a complicated skeleton of codimension one. The dynamics of one-fold symmetric patches for the two-dimensional aggregation equation has been studied by Hmidi and Li \cite{hl}.
 They showed that for domains with a suitable symmetry structure the solution converges weakly towards a finite measure supported in the union of disjoint segments lying in the real axis.

Since the arbitrary dimension problem is presently rather ambitious, we consider a toy model which corresponds to equation \eqref{jc100502} but just for dimension 1. In this case, the fundamental solution of the laplacian is simply \linebreak$N(x)=\frac{1}{2}\left|x\right|,$ and so the kernel is
$$K(x)=-N_x(x) = -\frac{1}{2}\text{sign}(x) = \begin{cases} 1/2&\text{if }x<0, \\ -1/2&\text{if }x\ge 0. \end{cases} $$

As we will see later, an advantage of considering this case is that the velocity $v(X(\alpha,t),t)$ is independent of $t$ and hence it can be computed at the initial time as $v_0(\alpha)=(K\ast \chi_{D_0})(\alpha)$. So the particle trajectories are ``straight lines'' as a function of $t$ in $\mathbb{R}\times [0,1)$ in the sense that they can be written as 
$$X(\alpha,t)=\alpha+v_0(\alpha)t.$$ In dimension 2 or bigger we apparently lose this nice property which is fundamental to develop explicit computations concerning the behaviour of the skeleton.

\subsection{Outline of the paper}
In Section \ref{skeopen} we prove the evolution of an open set towards a countable collection of Dirac deltas.  In Section \ref{skecompact} we consider the evolution of a compact set.  We prove that any compact set of vanishing Lebesgue measure is the skeleton of some well chosen compact set. In particular, Hausdorff dimensions can be distorted arbitrarily. Indeed, we prove in Theorem \ref{compact} a more general result involving measures.
\section{Open set}
\label{skeopen}
As we just said, we consider the one-dimensional aggregation equation. Explicitly, we have
\begin{equation}
 \begin{cases}
\rho_t + (\rho v)_x = 0,\\
v(\cdot,t) = - N_x \ast \rho(\cdot,t), \\
\rho(\cdot,0)=\rho_0=\chi_{\Omega_0},
\end{cases}
\label{1dimagg}
\end{equation}
where $\rho:\mathbb{R}\times \mathbb{R}^+\rightarrow \mathbb{R}$ and $v:\mathbb{R}\times \mathbb{R}^+\rightarrow \mathbb{R}$, and $N$ is the fundamental solution of the Laplace operator, i.e. $N(x) = \frac{1}{2}\left|x\right|$, and therefore $N_x(x)=\frac{1}{2}\sign(x)$. The following theorem, describing the evolution of a general bounded open set in $\mathbb{R}$ under equation \eqref{1dimagg}, contains the choice of a representative in $L^\infty(\mathbb{R})$ for the initial condition.

\begin{thm}
Let $\Omega_0 = \bigcup\limits_{i=1}^{\infty} I_i \subset \mathbb{R}$, with $I_i=(\alpha_i,\beta_i)$ pairwise disjoint, be a bounded open set. 
Let $\rho$ and $v$ be the solution of \eqref{1dimagg} with initial condition 
$\chi_{\Omega_0}$. Then, $\Omega_1 = \bigcup\limits_{i=1}^{\infty} \{x_i\}$ and
 if $\diff \mu_t = \rho(\cdot, t)\,\diff x$, we have $\mu_t \overset{w}{\longrightarrow} \mu_1= \sum\limits_{i=1}^{\infty} \abs{I_i} \delta_{x_i} $, where $\delta_{x_i}$ is the Dirac measure at $x_i$.
\label{thmskeleton1}
\end{thm}

\begin{proof}[Proof of Theorem \ref{thmskeleton1}]
In order to normalise a representative choice in $L^{\infty}(\mathbb R)$ for the initial condition, we assume without loss of generality that given $x\in I_j$, $y\in I_k$ with $j\neq k$, we have $\abs{I(x,y)\cap \Omega_0} < \abs{x-y},$
where $I(x,y)$ indicates the minimum interval containing $x$ and $y$.
In particular, this implies that instead of $(a_1,a_2) \cup (a_2, a_3)$ we will have $(a_1,a_3)$. 
It will also be seen from the proof, that if $[a, b]$ is the smallest closed interval containing $\Omega_0$ and $|\Omega_0| = 2L >0$ then the skeleton $\Omega_1$ is contained in $[a+L,b-L]$, because  it is easily verified that $v_0(a)=L$ and $v_0(b)=-L$.

First of all, let us see that under these assumptions, there exists spatial derivative for the trajectory map at least for  $\alpha \in \Omega_0$. Recall that, for any $0\le t<1$ the trajectory map $X(\cdot,t)$ is the unique homeomorphism solution to the ODE
$$\begin{cases} \frac{\diff X(\alpha,t)}{\diff t}=v(X(\alpha,t),t),\\ 
X(\alpha,0)=\alpha. \end{cases}$$
Hence, differentiating with respect to $\alpha$ the previous equation, we obtain
\begin{equation}
\begin{split}
\frac{\diff}{\diff \alpha} \left(\frac{\diff X(\alpha,t)}{\diff t}\right)&= \frac{\diff}{\diff \alpha}\left(v(X(\alpha,t),t)\right) = \\
&= \frac{\diff}{\diff \alpha}\left[\frac{1}{2} \int_{X(\alpha,t)}^{+\infty}\rho(y,t)\,\diff y -  \frac{1}{2} \int_{-\infty}^{X(\alpha,t)} \rho(y,t)\,\diff y\right] \\
&= -\rho(X(\alpha,t),t) \frac{\diff X(\alpha,t)}{\diff \alpha} = -\frac{1}{1-t}\rho_0(\alpha) \frac{\diff X(\alpha,t)}{\diff \alpha}.
\end{split}
\end{equation}

Interchanging derivatives we get
$$\frac{\diff}{\diff t}\left(\frac{\diff X(\alpha,t)}{\diff \alpha}\right) = -\frac{1}{1-t}\rho_0(\alpha) \left(\frac{\diff X(\alpha,t)}{\diff \alpha}\right). $$

Integrating and using that the homeomorphism $X(\cdot,0)$ is the identity map, we obtain the spatial derivative (for any  $\alpha \in \Omega_0$).
$$\frac{\diff X(\alpha,t)}{\diff \alpha} = 1-t. $$

Secondly, we can prove that the velocity of a particle is constant along the trajectory, this is,
\begin{equation}
v(X(\alpha,t),t) = v(\alpha,0)=:v_0(\alpha). 
\label{jc100504}
\end{equation}

This just requires a simple computation, involving a change of variable $y=X(\alpha',t)$.
\begin{equation*}
\begin{split}
v(X(\alpha,t),t) &=  \frac{1}{2} \int_{-\infty}^{+\infty}\sign(X(\alpha,t)-y)\rho(y,t)\,\diff y =\\ &= \frac{1}{2} \int_{-\infty}^{+\infty}\sign(X(\alpha,t)-X(\alpha',t))\rho(X(\alpha',t),t)\frac{\diff X(\alpha',t)}{\diff \alpha}\,\diff \alpha' = \\ 
&= \frac{1}{2} \int_{-\infty}^{+\infty}\sign(X(\alpha,t)-X(\alpha',t))\frac{1}{1-t}\rho_0(\alpha')\frac{\diff X(\alpha',t)}{\diff \alpha}\,\diff \alpha' = 
\\&= \frac{1}{2} \int_{\alpha' \in \Omega_0}\sign(X(\alpha,t)-X(\alpha',t))\frac{1}{1-t}\frac{\diff X(\alpha',t)}{\diff \alpha}\,\diff \alpha' = \\
&=  \frac{1}{2} \int_{\alpha' \in \Omega_0}\sign(X(\alpha,t)-X(\alpha',t))\frac{1}{1-t}(1-t)\,\diff \alpha'  = \\
&=  \frac{1}{2} \int_{\alpha' \in \Omega_0}\sign(X(\alpha,t)-X(\alpha',t))\,\diff \alpha' =\\
&= \frac{1}{2} \int_{\alpha' \in \Omega_0}\sign(\alpha-\alpha')\,\diff \alpha'=v_0(\alpha);
\end{split}
\end{equation*}
where we have used that $X(\alpha,t)-X(\alpha',t)$ and $\alpha-\alpha'$ have the same sign since $X(\cdot,t)$ is a non-decreasing homeomorphism. Then, by \eqref{jc100504} it is clear that all particle trajectory maps are straight lines. Indeed, for $0\le t<1$, we have

$$X(\alpha,t) = \alpha + \int_0^t v(X(\alpha,s),s)\,\diff s = \alpha + \int_0^t v_0(\alpha)\,\diff s = \alpha + v_0(\alpha) t .$$
Now we can check that any $x\in \left[\alpha_i,\beta_i\right]$ has the same limit point $$\lim_{t\to 1^-} X(x,t).$$
In fact, 
\begin{equation}
\begin{split}
v_0&(x) = \frac{1}{2}\int_x^{\infty} \rho_0(y)\,\diff y - \frac{1}{2}\int_{-\infty}^x \rho_0(y)\,\diff y= \\
&= \frac{1}{2}\left[ \int_{\alpha_i}^{\infty} \rho_0(y)\,\diff y - \int_{\alpha_i}^{x} \rho_0(y)\,\diff y - 
\int_{-\infty}^{\alpha_i} \rho_0(y)\,\diff y - 
\int_{\alpha_i}^{x} \rho_0(y)\,\diff y \right]= \\
&=v_0(\alpha_i)-\int_{\alpha_i}^{x} \rho_0(y)\,\diff y = v_0(\alpha_i)-(x-\alpha_i),
\end{split}
\end{equation}
where we have used the fact that $\rho_0\equiv 1$ in $(\alpha_i,x).$ Hence, for any $x \in \left[\alpha_i,\beta_i\right],$ we have
\begin{equation}
X(x,t) = x+\left(v_0(\alpha_i)-(x-\alpha_i)\right)t \stackrel{t\to 1^-}{\longrightarrow} \alpha_i+v_0(\alpha_i),
\label{eq1}
\end{equation}
which does not depend on the choice of $x$. From now on, we  denote the limit point for each interval $I_j=(\alpha_j, \beta_j)$ as $x_j:=\alpha_j+v_0(\alpha_j)$. Finally, we have to see the convergence of the measure $\mu_t$ defined as $\diff \mu_t = \rho(x,t)\,\diff x$ towards $ \mu_1= \sum_{i=1}^{\infty} \abs{I_i} \delta_{x_i}$.

In order to prove this, let $f$ be a continuous function on $\mathbb{R}$. Then, recall
$$\rho(x,t) = \sum_{i=1}^{\infty} \frac{1}{1-t} \chi_{(\alpha_{i,t},\beta_{i,t})},$$
where $\alpha_{i,t} = X(\alpha_i,t)$ and $\beta_{i,t}=X(\beta_i,t)$. Summing up, we have
$$\langle f, \mu_t\rangle =\frac{1}{1-t} \sum_{i=1}^{\infty} \int_{\alpha_{i,t}}^{\beta_{i,t}} f(x)\,\diff x.$$

Let $m_{i,t}:=\inf\limits_{x\in (\alpha_{i,t},\beta_{i,t})} f(x)$ and $M_{i,t}:=\sup\limits_{x\in (\alpha_{i,t},\beta_{i,t})} f(x)$. Then, it is clear that for any $i$,
$$\frac{1}{1-t} m_{i,t}(\beta_{i,t}-\alpha_{i,t})  \le \frac{1}{1-t} \int_{\alpha_{i,t}}^{\beta_{i,t}} f(x)\,\diff x \le \frac{1}{1-t} M_{i,t}(\beta_{i,t}-\alpha_{i,t}). $$

On the other hand, from \eqref{eq1} we have that
$$\beta_{i,t}-\alpha_{i,t} = (1-t)(\beta_i-\alpha_i) $$
and hence 
$$ m_{i,t}(\beta_{i}-\alpha_{i})  \le \frac{1}{1-t} \int_{\alpha_{i,t}}^{\beta_{i,t}} f(x)\,\diff x \le  M_{i,t}(\beta_{i}-\alpha_{i}). $$
Both the left and the right hand sides of the previous inequality clearly tend to the same value $f(x_i)(\beta_i-\alpha_i)$, by definition of $m_{i,t}$ and $M_{i,t}$. Therefore we have  
$$ \frac{1}{1-t} \int_{\alpha_{i,t}}^{\beta_{i,t}} f(x)\,\diff x \longrightarrow f(x_i)(\beta_i-\alpha_i),\,\,\,\,\text{
as }\,\,t\to 1^-.$$ Then 
$$\langle f,\mu_t\rangle \longrightarrow \langle f, \sum_{i=1}^{\infty} \left|I_i\right| \delta_{x_i}\rangle, $$
which proves the result.

\end{proof}

We can also formulate the reciprocal result. Any bounded countable collection of points is the skeleton of some initial open set.
\begin{thm}
Let $\{x_j\}_{j=1}^{\infty}$ a bounded countable collection of points such that $[c,d]$ is the smallest closed interval containing $\bigcup\limits_{j=1}^{\infty} \{x_j\}$.
Let $\mu_1 = \sum\limits_{j=1}^{\infty} c_j \delta_{x_j}$, for $c_j>0$ and such that $\sum_{j=1}^{\infty} c_j =2L$, where $\delta_{x_j}$ is the Dirac measure at $x_j$. Then there exists a bounded open set $\Omega_0\subseteq [c-L,d+L]$ such that
the solution of \eqref{1dimagg} with initial condition $\chi_{\Omega_0}$ satisfies the following: 
\begin{enumerate}[i)]
\item $\Omega_1 = \bigcup\limits_{i=1}^{\infty} \{x_i\}.$
\item for $\Omega_t$ and for the measure 
$\diff \mu_t = \rho(x,t)\,\diff x = \frac{1}{1-t}\chi_{\Omega_t}(x)\,\diff x $
we have that $\mu_t \overset{w}{\longrightarrow} \mu_1$ as $t\to 1^-$.
\end{enumerate} 
\label{thmskeleton2}
\end{thm}
\begin{proof}
For any $i\in \mathbb{N}$, let $l_i:=\sum_{x_j<x_i} c_j$ and define
$$\begin{cases}
a_i:=x_i+l_i- L,\\
b_i:=a_i+c_i.
\end{cases} $$
Define $\Omega_0:=\bigcup\limits_{i=1}^{\infty} (a_i,b_i)$. One easily checks that $\Omega_0$ satisfies the hypotheses of Theorem \ref{thmskeleton1} and hence by an straightforward
repetition of the argument in the proof of Theorem \ref{thmskeleton1} the theorem is proved.
\end{proof}

\begin{remark}
It is trivial to check that in Theorems \ref{thmskeleton1} and \ref{thmskeleton2}, the initial set $\Omega_0$ is not unique and the result also holds for any $\tilde{\Omega}_0$ such that
$$\Omega_0 \subseteq \tilde{\Omega}_0 \subseteq \bigcup\limits_{i=1}^{\infty} \overline{I_i} = \bigcup\limits_{i=1}^{\infty} [\alpha_i,\beta_i].$$
\end{remark}

\section{Compact set case}
\label{skecompact}
In the previous section we have seen that, when $\Omega_0$ is an open set, the limit measure is a countable combination of Dirac measures. Consequently, the Hausdorff dimension of this skeleton $\Omega_1$ is 0. Now, we shall prove that if we do not require the set to be open, we can obtain a skeleton of any Hausdorff dimension.
 Specifically, we shall prove that given a measure $\mu_1$ supported on a set $\Omega_1$ with zero length (but no necessarily having Hausdorff dimension equal to 0) we can construct a set $\Omega_0$ such that, if $\rho_0=\chi_{\Omega_0}$, then the solution $\rho$ of  \eqref{1dimagg} evolves towards $\mu_1$ as a measure.

\begin{thm} \label{compact}
 Let $K_1$ be a compact set with $\left|K_1\right|=0$ and let $\left[c,d\right]$  be the smallest closed interval containing $K_1$. Let $\mu_1$ be a measure with support equal to $K_1$ and $\mu_1(K_1)=2L$. Then, there exists a compact set $K_0$ with $\left|K_0\right|=2L$ and such that the solution $\rho(\cdot,t)$ to the  transport equation \eqref{1dimagg} with initial data $\rho_0 = \chi_{K_0}$ satisfies
$$\lim_{t\to 1^-} \rho(x,t)\diff x \overset{w}{\longrightarrow} \diff \mu_1. $$
\end{thm}

\begin{proof}
Since $K_1$ is compact, then the set $U_1=\left[c,d\right]\setminus K_1$ is open. Then it can be written as a numerable union of pairwise disjoint open intervals as 
$$U_1 = \bigcup\limits_{j=1}^{\infty} (a_{j,1},b_{j,1}).$$
For a point $x\in (a_{i,1},b_{i,1})$ we associate the following velocity (recall that in Theorem \ref{thmskeleton1} we saw that velocity is constant along trajectories)
\begin{equation}
v_i=\frac{1}{2}\left\{\mu_1\left(K_1\cap \left[\frac{a_{i,1}+b_{i,1}}{2},d\right] \right)- \mu_1\left(K_1\cap \left[c,\frac{a_{i,1}+b_{i,1}}{2}\right]\right)\right\}.
\label{3.1}
\end{equation}

Now we define 
$$\begin{cases}
a_{i,0} = a_{i,1}-v_i,\\
b_{i,0}=a_{i,0}+ (b_{i,1}-a_{i,1}) = b_{i,1}-v_i.
\end{cases}$$
and also we let $U_0 = \bigcup\limits_{i=1}^{\infty} (a_{i,0},b_{i,0})$ and set $a:=c-L$, $b:=d+L$. We define $K_0=[a,b]\setminus U_0$. 

The spirit of this procedure is as follows. We have observed in the proof of Theorem \ref{thmskeleton1} that the intervals where $\rho_0$ is 0 just move by keeping its length, because the velocity is the same for all points in the interval. What we have done here is keeping the length of the intervals in the complementary of $K_1$ and move them at the expected speed (constant for each interval) for them. So we get the right compact set.

Consider $\rho_0 = \chi_{K_0}$ ($\diff \mu_0 = \chi_{K_0}\,\diff x$ the Lebesgue measure restricted on $K_0$) and let $\rho(\cdot,t)$ be the solution to the  transport equation \eqref{1dimagg}. We have  
\begin{equation}
v(x)=(-\sign\ast \chi_{K_0})(x)=\frac{1}{2}\left\{ \left|K_0\cap (x,b)\right| - \left|K_0\cap(a,x)\right|\right\} .
\label{3.2}
\end{equation}
Observe that if $x\in U_0$ then $x\in (a_{i,0}, b_{i,0})$ for some index $i$ and therefore $v(x)=v_i$ as defined in \eqref{3.1}.
For $0\le t<1$, the flow provided by $v$ is
$$
X_t(x)=X(x,t)=\begin{cases}
x+tL&\text{if }x\le a,\\
x+v(x)t& \text{if }a<x<b, \\
x-tL &\text{if }b \le x,
\end{cases}
$$ 
a non decreasing homeomorphism in $\mathbb R$. Denote $K_t= X(K_0,t)$. On the other hand,
$$
\lim_{t\to 1^{-}}X_t(x)=X_1(x)=\begin{cases}
x+L&\text{if }x\le a,\\
x+v(x)& \text{if }a<x<b, \\
x-L &\text{if }b \le x,
\end{cases}
$$ 
is a surjective continuous map from $[a,b]$ to $[c,d]$. 
It is clear by construction that $U_1=X_1(U_0)$ and hence $K_1 = X_1(K_0)$ too. 

The solution of \eqref{1dimagg} is $\rho (x,t)=\dfrac{1}{1-t}\chi_{K_t}$ and $d \mu_t = \rho(x,t) dx$, $0\le t<1$. Given a Borel set $B$ of $\mathbb R$ one checks
$$
\mu_t(B) = \frac{1}{1-t}|K_t \cap B|= \mu_0(X_t^{-1}(B)),
$$
that is, $\mu_t$ is the image of $\mu_0$ under $X_t$, $\mu_t= (X_t)_{\#} \mu_0$.
It remains to check that $\mu_t$ converges weakly to $\mu_1$
when $t$ tends to 1. Let $g$ a continuous function on $\mathbb R$. Then, because $X_t \longrightarrow X_1$ continuously when $t\to 1$,
\begin{equation*}
\begin{split}
\int g d\mu_t &= \int g d (X_t)_{\#} \mu_0= \int (g\circ X_t)d\mu_0   \\
&= \int (g\circ X_t) \chi_{K_0}dx \longrightarrow\int (g\circ X_1) \chi_{K_0}dx
= \int g d (X_1)_{\#} \mu_0.
\end{split}
\end{equation*}
We have $\mu_t  \overset{w}{\longrightarrow} (X_1)_{\#} \mu_0$. Finally, we need to verify that $\mu_1 = (X_1)_{\#} \mu_0 $.

Given $y\in [c,d]$, set $X_1^{-1}(y)=\{x\in[a,b] : X_1(x)=y \}$.
We claim that $X_1^{-1}(y)$ is either a single element or an closed interval. Assume that $x_0 < x_1$ both in $X_1^{-1}(y)$ and then $x_0+v(x_0) = x_1+ v(x_1)=y$. Then by \eqref{3.2}
$$
x_1- x_0= v(x_0)-v(x_1)= |K_0 \cap (x_0,x_1)|.
$$
Consequently, for all $x\in[x_0, x_1]$ one has $x+v(x)= x_0+v(x_0)= y$ as we wanted. In fact, if $x^{+} = \sup X_1^{-1}(y)$ and $x^{-}= \inf X_1^{-1}(y)$ one has 
$X_1^{-1}(y)=[x^{-}, x^{+}]$. Now, again by construction, one has $\mu_1 (-\infty,y]= \mu_0 (-\infty , x^{+}]$ and $\mu_1 (-\infty , y) = \mu_0 (-\infty, x^{-})$. Playing with these two equations we get $\mu_0 (X_1^{-1}(J))= \mu_1 (J)$ for any interval J, whether open, closed, or semi-open. For instance, 
\begin{equation*}
\begin{split}
\mu_0 (X_1^{-1}[y_1, y_2]) 
&=\mu_0 (y_1^{-}, y_2^ {+}) = \mu_0 (-\infty, y_2^{+}) - \mu_0 (-\infty, y_1^{-}) \\
&= \mu_1(-\infty , y_2] -\mu_1 (-\infty , y_1)) \\
& = \mu_1[y_1,y_2] .
\end{split}
\end{equation*}
In conclusion $\mu_1 = (X_1)_{\#} \mu_0 $.

\end{proof}

\section*{Acknowledgements} 
The authors would like to warmly thank Joan Verdera for his useful comments about the writing of this article. They also acknowledge support by 2017-SGR-395 (AGAUR, Generalitat de Cata\-lunya) and PID2020-112881GB-100 (AEI, Ministerio de Ciencia e Innovaci\'on). J. Orobitg is supported by Severo Ochoa and Maria de Maeztu program for centers CEX2020-001084-M.


\bibliographystyle{amsplain}

\vspace{0.5cm}
{\small
	\begin{tabular}{@{}l}
		J.C.\ Cantero,\\ Departament de Matem\`{a}tiques, Universitat Aut\`{o}noma de Barcelona.\\
		J.\ Orobitg,\\
		Departament de Matem\`{a}tiques, Universitat Aut\`{o}noma de Barcelona,\\
		Centre de Recerca Matem\`atica, Barcelona, Catalonia.\\
		{\it E-mail:} {\tt cantero@mat.uab.cat},\,  {\tt orobitg@mat.uab.cat}
		
\end{tabular}}

\end{document}